\documentclass[12pt]{amsart}
\usepackage{amsmath,amsthm,amsfonts,amssymb,mathrsfs}
\date{\today}

\usepackage[cp1251]{inputenc}         

\usepackage[ukrainian]{babel} 
\usepackage{amsmath,amsthm,amsfonts,amssymb,mathrsfs}
\usepackage{eucal}

\input xymatrix
\xyoption{all}

\usepackage{color}

\usepackage{amssymb,cite}
\usepackage[colorlinks,plainpages,citecolor=magenta, linkcolor=blue, backref]{hyperref}

\usepackage{hyperref}

\usepackage{tikz}

\newtheorem{theorem}{Теорема}

\newtheorem{proposition}{Твердження}
\newtheorem{corollary}{Наслiдок}
\newtheorem{lemma}{Лема}
\theoremstyle{definition}
\newtheorem{example}{Приклад}
\newtheorem{remark}{Зауваження}

\newtheorem{definition}[theorem]{Означення}

\begin{document}

\title[Про топологiзацiю розширеної бiциклічноi напівгрупи]{Про топологiзацiю розширеної бiциклiчної напiвгрупи}

\author[Олег~Гутік, Маргарита Золотар, Олександра Лисецька]{Олег~Гутік, Маргарита Золотар, Олександра Лисецька}
\address{Львівський національний університет ім. Івана Франка, Університецька 1, Львів, 79000, Україна}
\email{oleg.gutik@lnu.edu.ua, marharyta.zolotar@lnu.edu.ua, oleksandra.lysetska@lnu.edu.ua}



\begin{abstract}
На розширеній біциклічній напівгрупі $\mathscr{C}_\mathbb{Z}$ побудовано недискретні напівгрупові $T_1$-то\-по\-ло\-гії. Знайдено топологічні умови, за виконання яких напівгрупова (трансляційно-неперервна) $T_1$-топологія на $\mathscr{C}_\mathbb{Z}$ є дискретною.

\bigskip
\noindent
\emph{Oleg Gutik, Marharyta Zolotar, Oleksandra Lysetska, \textbf{On topologization of the extended bicyclic semigroup}.}

\smallskip
\noindent
Non-discrete semigroup $T_1$-topologies on the extended bicyclic semigroup $\mathscr{C}_\mathbb{Z}$ are constructed. Also, we present topological conditions, under which a semigroup (shift-continuous) $T_1$-topology on $\mathscr{C}_\mathbb{Z}$ is discrete.

\end{abstract}

\maketitle


\section{Вступ}\label{section-1}

У цій праці ми користуємося термінологією з монографій \cite{Clifford-Preston-1961, Clifford-Preston-1967, Lawson-1998, Petrich-1984}.
Надалі у тексті множини натуральних, невід'ємних цілих і цілих чисел  позначатимемо через $\mathbb{N}$, $\omega$ i $\mathbb{Z}$, відповідно.
Надалі, якщо $X$~-- топологічний простір і $Y$~-- підпростір у $X$, то через $\operatorname{cl}_Y(A)$ й $\operatorname{int}_Y(A)$ будемо позначати замикання та внутрішність, відповідно, множини $A\subseteq Y$ у топологічному просторі $Y$.

Якщо $S$~--- напівгрупа, то її підмножину ідемпотентів позначатимемо через $E(S)$.  На\-пів\-гру\-па $S$ називається \emph{інверсною}, якщо для довільного її елемента $x$ існує єдиний елемент $x^{-1}\in S$ такий, що $xx^{-1}x=x$ та $x^{-1}xx^{-1}=x^{-1}$ \cite{Petrich-1984, Vagner-1952}. В інверсній напівгрупі $S$ вище означений елемент $x^{-1}$ називається \emph{інверсним до} $x$, а відображення $S\to S$, яке ставить кожному елементові напівгрупи $S$ інверсний до нього, називається \emph{інверсією}. \emph{В'язка}~-- це напівгрупа ідемпотентів, а \emph{напівґратка}~--- це комутативна в'язка.

Якщо $S$~--- напівгрупа, то на $E(S)$ визначено частковий порядок:
$
e\preccurlyeq f
$   тоді і лише тоді, коли
$ef=fe=e$.
Так означений частковий порядок на $E(S)$ називається \emph{при\-род\-ним}.

Означимо відношення $\preccurlyeq$ на інверсній напівгрупі $S$ так:
$
    s\preccurlyeq t
$
тоді і лише тоді, коли $s=te$, для деякого ідемпотента $e\in S$. Так означений частковий порядок назива\-єть\-ся \emph{при\-род\-ним част\-ковим порядком} на інверсній напівгрупі $S$~\cite{Vagner-1952}. Очевидно, що звуження природного часткового порядку $\preccurlyeq$ на інверсній напівгрупі $S$ на її в'язку $E(S)$ є при\-род\-ним частковим порядком на $E(S)$.

Нехай $X$, $Y$ і $Z$ -- топологічні простори. Відображення $f\colon X\times Y\to Z$, $(x,y)\mapsto f(x,y)$, називається:
\begin{itemize}
  \item[$(i)$] \emph{неперервним справа} [\emph{зліва}], якщо  $f$ неперервне по правій [лівій] змінній; тобто для кожної фіксованої точки $x_0\in X$ [$y_0\in Y$] відображення $Y\to Z$, $y\mapsto f(x_0,y)$ [$X\to Z$, $x\mapsto f(x,y_0)$] неперервне;
  \item[$(ii)$] \emph{нарізно неперервним}, якщо $f$ одночасно неперервне справа та неперервне зліва;
  \item[$(iii)$] \emph{неперервним}, якщо $f$ неперервне як відображення з добутку просторів $X\times Y$ у простір $Z$.
\end{itemize}

\begin{definition}[\cite{Carruth-Hildebrant-Koch=1983, Ruppert=1984}]
Нехай $S$ -- топологічний простір з визначеною на ньому напівгруповою операцією $\mu\colon S\times S\to S$, $(x,y)\mapsto \mu(x,y)=xy$. Тоді пара  $(S,\mu)$ називається:
\begin{itemize}
  \item[$(i)$]  \emph{правою топологічною напівгрупою}, якщо відображення $\mu$ неперервне справа, тобто всі внут\-рішні ліві зсуви $\lambda_s\colon S\to S$, $x\mapsto sx$, є неперервними відображеннями, $s\in S$;
  \item[$(ii)$] \emph{лівою топологічною напівгрупою}, якщо відображення $\mu$ неперервне зліва, тобто всі внут\-ріш\-ні праві зсуви $\rho_s\colon S\to S$, $x\mapsto xs$, є неперервними відображеннями, $s\in S$;
  \item[$(iii)$] \emph{напівтопологічною напівгрупою}, якщо відображення $\mu$ нарізно неперервне;
  \item[$(iv)$] \emph{топологічною напівгрупою}, якщо відображення $\mu$ неперервне.
\end{itemize}

Зазвичай ми опускаємо запис $\mu$ і пишемо просто $S$ замість $(S,\mu)$. Очевидно, що кожна топологічна напівгрупа також є напівтопологічною, і кожна напівтопологічна напівгрупа є одночасно правою та лівою топологічною напівгрупою.
\end{definition}

Топологія $\tau$ на напівгрупі $S$ називається:
\begin{itemize}
  \item \emph{напівгруповою}, якщо $(S,\tau)$ -- топологічна напівгрупа;
  \item \emph{інверсною}, якщо інверсія в $(S,\tau)$ неперевна;
  \item \emph{трансляційно-неперервною}, якщо $(S,\tau)$ -- напівтопологічна напівгрупа;
  \item \emph{неперервною зліва}, якщо $(S,\tau)$ -- ліва топологічна напівгрупа;
  \item \emph{неперервною справа}, якщо $(S,\tau)$ -- права топологічна напівгрупа.
\end{itemize}

Нагадаємо (див.  \cite[\S1.12]{Clifford-Preston-1961}, що \emph{біциклічною напівгрупою} (або \emph{біциклічним моноїдом}) ${\mathscr{C}}(p,q)$ називається напівгрупа з одиницею, породжена двоелементною мно\-жи\-ною $\{p,q\}$ і визначена єдиним  співвідношенням $pq=1$. Кожний елемент біциклічної напівгрупи ${\mathscr{C}}(p,q)$ єдиним чином зображається у вигляді $b^ia^j$, $i,j\in\omega$, а напівгрупова операція на ${\mathscr{C}}(p,q)$ визначається за формулою
\begin{equation*}
  b^ia^j\cdot b^ka^l=
  \left\{
    \begin{array}{ll}
      b^{i-j+k}a^l, & \hbox{якщо~} j<k; \\
      b^ia^l        , & \hbox{якщо~} j=k; \\
      b^ia^{j-k+l}, & \hbox{якщо~} j>k,
    \end{array}
  \right.
\end{equation*}
$i,j,k,l\in\omega$.

Біциклічна на\-пів\-група відіграє важливу роль у теорії
на\-півгруп. Так, зокрема, класична теорема О.~Ан\-дерсена \cite{Andersen-1952}  стверджує, що {($0$-)}прос\-та напівгрупа з (ненульовим) ідем\-по\-тен\-том є цілком {($0$-)простою} тоді і лише тоді, коли вона не містить ізоморфну копію бі\-циклічного моноїда. Різні розширення та узагальнення біциклічного моноїда вводилися раніше різ\-ни\-ми авторами \cite{Gutik-Mykhalenych-2020, Fortunatov-1976, Fotedar-1974, Fotedar-1978, Gutik-Pagon-Pavlyk=2011, Warne-1967}. Такими, зокрема, є конструкції Брука та Брука--Рейлі занурення напівгруп у прості та описання інверсних біпростих і $0$-біпростих $\omega$-напівгруп \cite{Bruck-1958, Reilly-1966, Warne-1966, Gutik-2018}.

\begin{remark}\label{remark-10}
Легко бачити, що біциклічний моноїд ${\mathscr{C}}(p,q)$ ізоморфний напівгрупі, заданій на множині $\boldsymbol{B}_{\omega}=\omega\times\omega$ з напівгруповою операцією
\begin{equation}\label{eq-1}
  (i_1,j_1)\cdot(i_2,j_2)=
\left\{
  \begin{array}{ll}
    (i_1-j_1+i_2,j_2), & \hbox{якщо~} j_1\leqslant i_2;\\
    (i_1,j_1-i_2+j_2), & \hbox{якщо~} j_1\geqslant i_2,
  \end{array}
\right.
\end{equation}
стосовно відображення $\mathfrak{I}\colon {\mathscr{C}}(p,q)\to \boldsymbol{B}_{\omega}$, $b^ia^j\mapsto (i,j)$, $i,j\in\omega$.
\end{remark}

Загальновідомо, що топологічна алгебра вивчає вплив топологічних властивостей своїх об'єктів на їхні алгебраїчні властивості та вплив алгебраїчних властивостей своїх об'єктів на їхні топологічні властивості. У топологічній алгебрі є дві основні задачі: задача недискретної топологізації та задача вкладення в об'єкти з деякими топологічно-алгебраїчними властивостями.

\smallskip

У математичній літературі питання про недискретну (гаусдорфову) топологізацію було поставлене Марковим \cite{Markov=1945}.
Понтрягін навів добре відомі умови бази в одиниці групи для її недискретної топологізації (див.  теорему~3.9 у \cite{Pontryagin=1966}).
Різні автори уточнили питання Маркова: \emph{чи можна ізоморфно занурити задану нескінченну групу $G$, наділену недискретною груповою топологією, у компактну тополо\-гіч\-ну групу?} Знову ж таки, для довільної абелевої групи $G$ відповідь ствердна, але існує неабелева топологічна група, яку не можна ізоморфно занурити в жодну компактну топологічну групу ({див. розділ~9 у \cite{HBSTT-1984}}).

\smallskip

Також, Ольшанський \cite{Olshansky=1980} побудував нескінченну зліченну групу $G$ таку, що кожна групова гаусдорфова топологія на $G$ дискретна. Тайманов у 1973 році побудував комутативну напівгрупу $\mathfrak{T}$, яка допускає лише дискретну напівгрупову гаусдорфову топологію \cite{Taimanov=1973}, та вказав у \cite{Taimanov=1975} достатні умови на комутативну напівгрупу, щоб на ній існувала недискретна гаусдорфова напівгрупова топологія. У \cite{Gutik=2016} доведено, що кожна $T_1$-топологія з неперервними зсувами на $\mathfrak{T}$ є дискретною. Однак, вперше напівгрупу на якій кожна гаусдорфова напівгрупова топологія є дискретною було опубліковано в праці Ебергарта та Селдена \cite{Eberhart-Selden=1969}, і нею виявилась біциклічна напвгрупа ${\mathscr{C}}(p,q)$. Берт\-ман і Вест у \cite{Bertman-West=1976} довели, що кожна гаусдорфова трансляційно-неперервна топологія на ${\mathscr{C}}(p,q)$ дискретна. У \cite{Chornenka-Gutik=2023} побудовано приклади недискретних інверсних напівгрупових і трансляційно-неперервних $T_1$-топологій на біциклічному моноїді ${\mathscr{C}}(p,q)$ та вказано умови за виконання яких напівгрупова (трансляційно-неперервна) $T_1$-топологія на ${\mathscr{C}}(p,q)$ є дискретною. Недискретна топологізація піднапівгруп біциклічного моноїда досліджувалася в праці \cite{Chornenka-Gutik=2026}. У \cite{Gutik-2024} доведено, що кожна неперервна зліва (неперервна справа) гаусдорфова топологія на верхній піднапівгрупі ${\mathscr{C}}_+(p,q)$ (нижній піднапівгрупі ${\mathscr{C}}_-(p,q)$) біциклічного моноїда дискретна, а також побудовано на ${\mathscr{C}}_+(p,q)$ (${\mathscr{C}}_-(p,q)$) неперервну справа (неперервну зліва) гаусдорфову топологію. Крім того в \cite{Gutik-2026} доведено, що напівгрупа ${\mathscr{C}}_+(p,q)$ (${\mathscr{C}}_-(p,q)$) містить континуум неізоморфних піднапівгруп $S_\alpha$ ($\alpha\in\mathfrak{c}$) з такими властивостями: $(i)$~кожна гаусдорфова неперервна зліва (неперервна справа) топологія на $S_\alpha$ дискретна; $(ii)$~на $S_\alpha$ існує недискретна гаусдорфова  неперервна справа (неперервна зліва) топологія; $(iii)$~існує компактна гаусдофова топологічна напівгрупа, яка містить напівгрупу ${\mathscr{C}}_+(p,q)$ (${\mathscr{C}}_-(p,q)$).

Ворн у \cite{Warne-1967}, досліджуючи структуру біпростих інверсних напівгруп, піднапівгрупа ідемпотентів яких ізоморфна напівґратці $(\mathbb{Z},\max)$, ввів \emph{розширену біциклічну напівгрупу} як множину $\mathscr{C}_\mathbb{Z}=\mathbb{Z}\times\mathbb{Z}$ з напівгруповою операцією, визначеною за формулою \eqref{eq-1}.

Біциклічний моноїд ${\mathscr{C}}(p,q)$ та розширена біциклічна напівгрупа $\mathscr{C}_\mathbb{Z}$ мають багато спільних алгебричних властивостей. Зокрема ці дві напівгрупи біпрості й інверсні, а їхна піднапівгрупи ідемпотентів є лінійно впорядкованими нещільними множинами стосовно природного часткового порядку. Тому виникає природне запитання: \emph{які топологічні властивості схожі до біциклічного моноїда має розширена біциклічна напівгрупа як тополого-алгебрична структура?} Перший результат у цьому напрямку було отримано в праці \cite{Fihel-Gutik-2011}, де було доведено, що кожна транслційно-неперервна гаусдорфова топологія на $\mathscr{C}_\mathbb{Z}$ дискретна.

Ця праця частково відповідає на поставлене вище запитання. На $\mathscr{C}_\mathbb{Z}$ побудовано недискретну локально компактну $T_1$-топологію, а також неперервну зліва (справа) берівську $T_1$-топологію. Також вказано умови, за виконання яких $T_1$-топологія на напівгрупі $\mathscr{C}_\mathbb{Z}$ є дискретною.

\section{Приклади недискретних топологій на $\mathscr{C}_\mathbb{Z}$}\label{section-2}

\begin{example}\label{example-2.1}
Визначимо топологію $\tau_1$ на $\mathscr{C}_\mathbb{Z}$ так. Нехай $(i,j)$~-- довільний елемент напівгрупи $\mathscr{C}_\mathbb{Z}$. Для кожного натурального числа $n$ позначимо $U_n(i,j)=\{(i,j)\}\cup\{(s,t)\in\mathscr{C}_\mathbb{Z}\colon s,t\geqslant n\}$. Нехай $\mathscr{B}_1(i,j)=\{U_n(i,j)\colon n\in\mathbb{N}\}$ -- система відкритих околів у точці $(i,j)$. Тоді сім'я  $\mathscr{B}_1=\bigcup_{i,j\in\mathbb{Z}}\mathscr{B}_1(i,j)$ задовольняє умови (BP1)--(BP3) з \cite{Engelking=1989}, а отже, породжує деяку топологію на $\mathscr{C}_\mathbb{Z}$. Позначимо цю топологію через $\tau_1$.

Очевидно, що  $\tau_1$ -- недискретна $T_1$-топологія на  $\mathscr{C}_\mathbb{Z}$.
\end{example}

\begin{proposition}\label{proposition-2.2}
$(\mathscr{C}_\mathbb{Z},\tau_1)$~-- топологічна інверсна напівгрупа.
\end{proposition}

\begin{proof}
Зафіксуємо довільне натуральне число $n$. З визначення напівгрупової операції на $\mathscr{C}_\mathbb{Z}$ випливає, що для довільних елементів $(i_1,j_1)$ та  $(i_2,j_2)$ напівгрупи $\mathscr{C}_\mathbb{Z}$ виконуються включення
\begin{equation*}
  U_m(i_1,j_1)\cdot U_m(i_2,j_2)\subseteq U_n(k,l) \qquad \hbox{i} \qquad (U_n(i_1,j_1))^{-1}\subseteq U_n(j_1,i_1),
\end{equation*}
де $(k,l)=(i_1,j_1)\cdot(i_2,j_2)$ i $m\geqslant 2n+|i_1|+|j_1|+|i_2|+|j_2|$. Отже, напівгрупова операція та інверсія неперервні в $(\mathscr{C}_\mathbb{Z},\tau_1)$.
\end{proof}

Для природного часткового порядку $\preccurlyeq$ на розширеній біциклічній напівгрупі $\mathscr{C}_\mathbb{Z}$ і для довільного елемента $(i,j)\in\mathscr{C}_\mathbb{Z}$ означимо:
\begin{align*}
  {\uparrow_{\preccurlyeq}}(i,j)         &=\left\{(s,t)\in{\mathscr{C}}_\mathbb{Z}\colon (i,j)\preccurlyeq (s,t)\right\}; \\
  {\downarrow_{\preccurlyeq}}(i,j)       &=\left\{(s,t)\in{\mathscr{C}}_\mathbb{Z}\colon (s,t)\preccurlyeq (i,j)\right\}; \\
  {\updownarrow_{\preccurlyeq}}(i,j)     &={\uparrow_{\preccurlyeq}}(i,j)\cup {\downarrow_{\preccurlyeq}}(i,j); \\
  {\downarrow_{\preccurlyeq}^\circ}(i,j) &={\downarrow_{\preccurlyeq}}(i,j)\setminus\left\{(i,j)\right\}.
\end{align*}

Наступне твердження описує природний частковий порядок $\preccurlyeq$ на розширеній бі\-цик\-лічній напівгрупі $\mathscr{C}_\mathbb{Z}$ і випливає з означення відношення $\preccurlyeq$ і твердження~2.1$(vi)$ \cite{Fihel-Gutik-2011}.

\begin{proposition}\label{proposition-2.3}
Для елементів $(i,j)$ і $(s,t)$ розширеної біциклічної напівгрупи $\mathscr{C}_\mathbb{Z}$ такі умови еквівалентні:
\begin{enumerate}
  \item[$(i)$]   $(i,j)\preccurlyeq (s,t)$;
  \item[$(ii)$]  $i\geqslant s$ та $i-j=s-t$;
  \item[$(iii)$] $j\geqslant t$ та $i-j=s-t$.
\end{enumerate}
\end{proposition}

Надалі нам часто буде потрібний такий простий результат:

\begin{lemma}\label{lemma-2.4}
Нехай $S$~-- інверсна ліва (права) $T_1$-топологічна напівгрупа. Тоді для довільного елемента $x$ напівгрупи $S$ множина ${\uparrow_{\preccurlyeq}}x$ є замкненою в просторі $S$
.\end{lemma}

\begin{proof}
 З  леми~1.4.6 з \cite{Lawson-1998} випливає, що
\begin{equation*}
  {\uparrow_{\preccurlyeq}}x=\{s\in S\colon s\cdot xx^{-1}=x\}=\{s\in S\colon xx^{-1}\cdot s=x\},
\end{equation*}
а отже, ${\uparrow_{\preccurlyeq}}x=\rho_{xx^{-1}}^{-1}(x)=\lambda_{xx^{-1}}^{-1}(x)$.
Позаяк праві (ліві) зсуви в $S$ неперервні i $S$~-- $T_1$-простір, то множина ${\uparrow_{\preccurlyeq}}x=\rho_{xx^{-1}}^{-1}(x)$ (${\uparrow_{\preccurlyeq}}x=\lambda_{xx^{-1}}^{-1}(x)$)~-- замкнена в $S$.
\end{proof}

\begin{example}\label{example-2.4}
Визначимо топологію $\tau_2$ на $\mathscr{C}_\mathbb{Z}$ так. Для довільного елемента $(i,j)$ напівгрупи $\mathscr{C}_\mathbb{Z}$ і для довільного невід'ємного цілого числа $n$ позначимо $V_n(i,j)=\{(i,j)\}\cup{\downarrow_{\preccurlyeq}^\circ}(i+n,j+n)$. Нехай $\mathscr{B}_2(i,j)=\{V_n(i,j)\colon n\in\mathbb{N}\}$ -- система відкритих околів у точ\-ці $(i,j)$. Тоді сім'я  $\mathscr{B}_2=\bigcup_{i,j\in\mathbb{Z}}\mathscr{B}_2(i,j)$ задовольняє умови (BP1)--(BP3) з \cite{Engelking=1989}, а отже, породжує деяку топологію на $\mathscr{C}_\mathbb{Z}$. Позначимо цю топологію через $\tau_2$.

Із побудови топології $\tau_2$ на $\mathscr{C}_\mathbb{Z}$ випливає, що $(\mathscr{C}_\mathbb{Z},\tau_2)$~-- $T_1$-простір. Очевидно, що кожний елемент $V_n(i,j)$ сім'ї $\mathscr{B}_2(i,j)$ є компактною підмножиною в $(\mathscr{C}_\mathbb{Z},\tau_2)$, а отже, простір $(\mathscr{C}_\mathbb{Z},\tau_2)$ -- локально компактний.
\end{example}

\begin{proposition}\label{proposition-2.5}
$(\mathscr{C}_\mathbb{Z},\tau_2)$~-- топологічна інверсна напівгрупа.
\end{proposition}

\begin{proof}
З визначення напівгрупової операції на напівгрупі $\mathscr{C}_\mathbb{Z}$ випливає, що для довільних елементів $(i_1,j_1)$ та $(i_2,j_2)$ напівгрупи $\mathscr{C}_\mathbb{Z}$, для довільного невід'ємного цілого числа $n$ і натурального числа $m\geqslant 2n+|j_1|+|i_2|$ виконуються включення
\begin{equation*}
  V_m(i_1,j_1)\cdot V_m(i_2,j_2)\subseteq V_n(k,l) \qquad \hbox{i} \qquad (V_n(i_1,j_1))^{-1}\subseteq V_n(j_1,i_1),
\end{equation*}
де $(k,l)=(i_1,j_1)\cdot(i_2,j_2)$. Отже, $(\mathscr{C}_\mathbb{Z},\tau_2)$~-- топологічна інверсна напівгрупа.
\end{proof}

Нагадаємо \cite{Engelking=1989}, що топологічний простір $X$ називається \emph{зліченно компактним}, якщо кожне зліченне відкрите покриття простору $X$ містить скінченне підпокриття.

У \cite{Chornenka-Gutik=2023} доведено, що на біциклічній напівгрупі ${\mathscr{C}}(p,q)$ ко-скінченна топологія є транс\-ля\-цій\-но-неперервною та інверсія стосовно неї на ${\mathscr{C}}(p,q)$ також є неперервною, а отже, на біциклічній напівгрупі існує компактна трансляційно-неперервна $T_1$-топологія. З твердження~\ref{proposition-2.6} випливає, що аналогічне твердження не виконується для розширеної біциклічної напівгрупи.

\begin{proposition}\label{proposition-2.6}
Не існує зліченно компактної неперервної зліва (справа) $T_1$-топології на $\mathscr{C}_\mathbb{Z}$.
\end{proposition}

\begin{proof}
Нехай $\tau$~-- неперервна зліва $T_1$-топологія на $\mathscr{C}_\mathbb{Z}$. Для довільного елемента $(i,j)$ напівгрупи $\mathscr{C}_\mathbb{Z}$ позначимо через $\rho_{(i,j)}$ правий внутрішній зсув у напівгрупі $\mathscr{C}_\mathbb{Z}$ на елемент $(i,j)$, тобто $\rho_{(i,j)}(x,y)=(x,y)\cdot (i,j)$ для всіх $(x,y)\in\mathscr{C}_\mathbb{Z}$. Оскільки  $(\mathscr{C}_\mathbb{Z},\tau)$~-- $T_1$-простір, то за лемою~\ref{lemma-2.4}, ${\uparrow_{\preccurlyeq}}(i,i)=\rho^{-1}_{(i,i)}(i,i)$~-- зам\-кнена підмножина в $(\mathscr{C}_\mathbb{Z},\tau)$, для довільного ідемпотента $(i,i)$ напівгрупи $\mathscr{C}_\mathbb{Z}$. Тоді $\mathscr{U}=\{U_i=\mathscr{C}_\mathbb{Z}\setminus {\uparrow_{\preccurlyeq}}(i,i)\colon i\in\mathbb{Z}\}$~-- зліченне відкрите покриття простору $(\mathscr{C}_\mathbb{Z},\tau)$, яке не містить скінченне підпокриття.

У випадку неперервної справа топології на $\mathscr{C}_\mathbb{Z}$ доведення аналогічне.
\end{proof}

\begin{lemma}\label{lemma-2.7}
Для довільних елементів $(i,j)$ i $(k,l)$ розширеної біциклічної напівгрупи $\mathscr{C}_\mathbb{Z}$ виконуються такі включення
\begin{equation*}
  {\uparrow_{\preccurlyeq}}(i,j)\cdot (k,l)\subseteq {\uparrow_{\preccurlyeq}}(m,n)\qquad \mbox{i} \qquad
  (i,j)\cdot{\uparrow_{\preccurlyeq}}(k,l)\subseteq{\uparrow_{\preccurlyeq}}(m,n),
\end{equation*}
де $(m,n)=(i,j)\cdot(k,l)$.
\end{lemma}

\begin{proof}
З твердження~\ref{proposition-2.3} випливає рівність
\begin{equation}\label{eq-2.1}
  {\uparrow_{\preccurlyeq}}(p,q)=\{(p-s,q-s)\in\mathscr{C}_\mathbb{Z}\colon s\in\omega \},
\end{equation}
для довільного елемента $(p,q)\in \mathscr{C}_\mathbb{Z}$.

Зафіксуємо довільні елементи $(i,j)$ i $(k,l)$ напівгрупи $\mathscr{C}_\mathbb{Z}$.

Розглянемо можливі випадки.

Якщо $k\geqslant j$, то $(i,j)\cdot (k,l)=(i-j+k,l)$ і для довільного невід'ємного цілого числа $s$ маємо, що
\begin{align*}
  (i-s,j-s)\cdot (k,l)&=(i-s-(j-s)+k,l)= \\
   &=(i-j+k,l)=(m,n),
\end{align*}
а отже, з рівності \eqref{eq-2.1} випливає, що ${\uparrow_{\preccurlyeq}}(i,j)\cdot (k,l)\subseteq {\uparrow_{\preccurlyeq}}(m,n)$.

Якщо $k<j$, то $(i,j)\cdot (k,l)=(i,j-k+l)$ і для довільного невід'ємного цілого числа $s$ маємо, що
\begin{align*}
  (i-s,j-s)\cdot (k,l)&=
  \left\{
    \begin{array}{ll}
      (i-s-(j-s)+k,l), & \hbox{якщо~} j-s\leqslant k;\\
      (i-s,j-s-k+l),   & \hbox{якщо~} j-s>k
    \end{array}
  \right.=
  \\
   &=
  \left\{
    \begin{array}{ll}
      (i-j+k,l),     & \hbox{якщо~} j-s\leqslant k;\\
      (i-s,j-s-k+l), & \hbox{якщо~} j-s>k.
    \end{array}
  \right.
\end{align*}
Тоді
\begin{align*}
  (i-j+k,l)\cdot (j-k+l,j-k+l)&=(i-j+k-l+(j-k+l),j-k+l)= \\
    &=(i-j+k-l+j-k+l,j-k+l)= \\
    &=(i,j-k+l)
\end{align*}
i
\begin{align*}
  (i-s,j-s-k+l){\cdot}(j-k+l,j-k+l)&=(i{-}s{-}(j{-}s{-}k+l)+j{-}k+l,j{-}k+l)= \\
   &=(i-s-j+s+k-l+j-k+l,j-k+l)=\\
    &=(i,j-k+l)
\end{align*}
для довільного невід'ємного числа $s$. З  леми~1.4.6 з \cite{Lawson-1998} випливає, що $(i-s,j-s)\cdot(k,l)\in {\uparrow_{\preccurlyeq}}(m,n)$, а отже, скориставшись рівністю \eqref{eq-2.1}, отримуємо включення ${\uparrow_{\preccurlyeq}}(i,j)\cdot(k,l)\subseteq {\uparrow_{\preccurlyeq}}(m,n)$.

Доведення включення $(i,j)\cdot{\uparrow_{\preccurlyeq}}(k,l)\subseteq {\uparrow_{\preccurlyeq}}(m,n)$ аналогічне.
\end{proof}

Нагадаємо, що топологічний простір $X$ називається \emph{берівським}, якщо для довільної послідовності $A_1, A_2,\ldots, A_i,\ldots$ щільних відкритих підмножин простору $X$ перетин $\displaystyle\bigcap_{i=1}^\infty A_i$ є щільною підмножиною в $X$ \cite{Haworth-McCoy=1977}.

У праці \cite{Chornenka-Gutik=2026} доведено, що кожна ліво-неперевна (право-неперервна) гаусдорфова топологія на біциклічний напівгрупі дискретна. З прикладу~\ref{example-2.8} випливає, що для розширеної біциклічної напівгрупи аналогічне твердження не виконується.

\begin{example}\label{example-2.8}
Топологія $\tau_\mathrm{B}$ на розширеній біциклічній напівгрупі $\mathscr{C}_\mathbb{Z}$ породжується передбазою $\mathscr{P}_\mathrm{B}$, яка складається з таких сімей:
\begin{itemize}
  \item $\mathscr{P}_0=\left\{\{(i,j)\}\colon i\in\mathbb{Z}\setminus\mathbb{N}, j\in\mathbb{Z} \right\}$;
  \item $\mathscr{P}_1=\left\{\mathscr{C}_\mathbb{Z}\setminus \{(i,j)\}\colon i,j\in\mathbb{Z}\right\}$;
  \item $\mathscr{P}_2=\left\{\mathscr{C}_\mathbb{Z}\setminus {\uparrow_{\preccurlyeq}}(i,j)\colon i,j\in\mathbb{Z}\right\}$;
  \item $\mathscr{P}_3=\left\{O(i,j)=\bigcup_{k\in\omega} {\uparrow_{\preccurlyeq}}(i+k,j)\colon i,j\in\mathbb{Z}\right\}$.
\end{itemize}

Очевидно, що $\tau_\mathrm{B}$~-- недискретна топологія на $\mathscr{C}_\mathbb{Z}$. Також усі точки множини
\begin{equation*}
\mathrm{IZ}=\{(i,j)\colon i\in\mathbb{Z}\setminus\mathbb{N}, j\in\mathbb{Z}\}
\end{equation*}
є ізольованими в просторі $(\mathscr{C}_\mathbb{Z},\tau_\mathrm{B})$, а оскільки $\mathscr{P}_1\subseteq \mathscr{P}_\mathrm{B}$, то $\tau_\mathrm{B}$~-- $T_1$-топологія. З визначення передбази $\mathscr{P}_\mathrm{B}$ випливає, що для довільних $i\in\mathbb{Z}\setminus\mathbb{N}$ i  $j\in\mathbb{Z}$ сім'я $\mathscr{B}_{\mathrm{B}}(i,j)=\{U_p(i,j)\colon p\in\mathbb{N}\}$, де $U_p(i,j)=\{(i,j)\}\cup\bigcup_{s=p}^\infty{\uparrow_{\preccurlyeq}}(i+s,j)$, визначає базу топології $\tau_\mathrm{B}$ у точці $(i,j)$. За твердженням~1.3.5 з \cite{Engelking=1989}, $\mathrm{IZ}$~-- відкрита, щільна підмножина в просторі $(\mathscr{C}_\mathbb{Z},\tau_\mathrm{B})$, а отже, за  твердженням~1.18 і теоремою 1.15 з монографії \cite{Haworth-McCoy=1977} топологічний простір є берівським.
\end{example}

\begin{proposition}\label{proposition-2.9}
$(\mathscr{C}_\mathbb{Z},\tau_\mathrm{B})$~-- ліва топологічна напівгрупа.
\end{proposition}

\begin{proof}
Зафіксуємо довільні елементи $(i,j)$ i $(k,l)$ напівгрупи $\mathscr{C}_\mathbb{Z}$. З визначення топології $\tau_\mathrm{B}$ випливає, що достатньо розглянути лише випадок $(i,j)\cdot(k,l)$, де $i\in\mathbb{N}$.

Зафіксуємо довільний відкритий базовий окіл $U_p(m,n)$ точки
\begin{equation*}
  (m,n)=(i,j)\cdot(k,l)=
  \left\{
    \begin{array}{ll}
      (i-j+k,l), & \hbox{якщо~} j\leqslant k;\\
      (i,j-k+l), & \hbox{якщо~} j>k
    \end{array}
  \right.
\end{equation*}
у топологічному просторі $(\mathscr{C}_\mathbb{Z},\tau_\mathrm{B})$. Оскільки для довільного натурального числа $s$ виконується рівність
\begin{equation*}
\begin{split}
  (i+s,j)\cdot(k,l)&=
  \left\{
    \begin{array}{ll}
      (i+s-j+k,l), & \hbox{якщо~} j\leqslant k;\\
      (i+s,j-k+l), & \hbox{якщо~} j>k
    \end{array}
  \right.=\\
    &=(m+s,n),
\end{split}
\end{equation*}
то з леми~\ref{lemma-2.7} випливає, що ${\uparrow_{\preccurlyeq}}(i+s,j)\cdot(k,l)\subseteq {\uparrow_{\preccurlyeq}}(m+s,n)$. Звідси отримуємо, що для довільного натурального числа $p$ виконується включення $U_p(i,j)\cdot(k,l)\subseteq U_p(m,n)$, а отже, праві зсуви в $(\mathscr{C}_\mathbb{Z},\tau_\mathrm{B})$ неперервні.
\end{proof}

Нехай $X$~-- топологічнй простір й $Y$~-- підпростір у $X$. Будемо говорити, що простір $Y$ є \emph{квазі-регулярним} у точці $x\in Y$, якщо для довільного відкритого окулу $U(x)$ точки $x$ у просторі $Y$ існує відкрита непорожна підмножина $V$ у $Y$ така, що $\operatorname{cl}_Y(V)\subseteq U(x)$. Топологічний простір $X$ називається \emph{квазі-регуляним}, якщо $X$ -- квазі-регулярний у кожній своїй точці.

\begin{remark}\label{remark-2.10}
\begin{enumerate}
  \item За твердженням~1.30 з \cite{Haworth-McCoy=1977} кожний зліченний берівський $T_1$-прос\-тір міс\-тить нескіченну підмножину ізольованих у ньому точок, а отже, є квазі-ре\-гу\-ляр\-ним. Із вище сказаного випливає, що топологічний простір $(\mathscr{C}_\mathbb{Z},\tau_\mathrm{B})$ є квазі-регулярним.

  \item Кожна неперервна зліва інверсна $T_1$-топологія на розширеній біциклічній напівгрупі $\mathscr{C}_\mathbb{Z}$ є неперервною справа, а отже, за теоремою~1 з \cite{Gutik-Maksymyk=2016} є дискретною.

  \item Оскільки інверсія $\operatorname{inv}\colon \mathscr{C}_\mathbb{Z}\to\mathscr{C}_\mathbb{Z}$, $(i,j)\mapsto (j,i)$ є анти-ізоморфізмом розширеної біциклічної напівгрупи, то дуальна топологія $\tau_\mathrm{B}^{\mathrm{d}}=\{U\subseteq \mathscr{C}_\mathbb{Z}\colon \operatorname{inv}(U)\in \tau_\mathrm{B}\}$ до топології $\tau_\mathrm{B}$ є неперервною справа берівською $T_1$-топологією на $\mathscr{C}_\mathbb{Z}$.
\end{enumerate}
\end{remark}

\section{Коли топологія на $\mathscr{C}_\mathbb{Z}$ дискретна?}

\begin{lemma}\label{lemma-3.1}
\begin{enumerate}
  \item[$(i)$] Для довільних елементів $(i,j)$, $(k,l)$, $(m,n)$ i $(p,q)$ напівгрупи $\mathscr{C}_\mathbb{Z}$ з умови $(p,q)=(i,j)\cdot(k,l)\cdot(m,n)$ випливає рівність $p-q=i-j+k-l+m-n$.

  \item[$(ii)$] Елементіи $(i,j)$ і $(k,l)$ напівгрупи $\mathscr{C}_\mathbb{Z}$ порівняльні стосовно природного часткового порядку на $\mathscr{C}_\mathbb{Z}$ тоді і лише тоді, коли $i-j=k-l$.

  \item[$(iii)$] Якщо $(i_0,i_1)$, $(j_1,j_0)$ та $(i_0,j_0)$~-- фіксовані елементи напівгрупи $\mathscr{C}_\mathbb{Z}$, то $(x,y)\in\mathscr{C}_\mathbb{Z}$ є роз\-в'яз\-ком рівняння
      \begin{equation}\label{eq-3.1}
        (i_0,i_1)\cdot(x,y)\cdot(j_1,j_0)=(i_0,j_0)
      \end{equation}
      тоді і лише тоді, коли $(x,y)\in{\uparrow_{\preccurlyeq}}(i_1,j_1)$.

  \item[$(iv)$] Якщо $(j_1,j_0)$ та $(i_0,j_0)$~-- фіксовані елементи напівгрупи $\mathscr{C}_\mathbb{Z}$, то $(x,y)\in\mathscr{C}_\mathbb{Z}$ є роз\-в'яз\-ком рівняння $(x,y)\cdot(j_1,j_0)=(i_0,j_0)$ тоді і лише тоді, коли $(x,y)\in{\uparrow_{\preccurlyeq}}(i_0,j_1)$.

  \item[$(v)$] Якщо $(i_0,i_1)$ та $(i_0,j_0)$~-- фіксовані елементи напівгрупи $\mathscr{C}_\mathbb{Z}$, то $(x,y)\in\mathscr{C}_\mathbb{Z}$ є роз\-в'яз\-ком рівняння $(i_0,i_1)\cdot(x,y)=(i_0,j_0)$ тоді і лише тоді, коли $(x,y)\in{\uparrow_{\preccurlyeq}}(i_1,j_0)$.
\end{enumerate}
\end{lemma}

\begin{proof}
$(i)$   За твердженням~2.1$(iv)$ з \cite{Fihel-Gutik-2011} з умови $(a,b)=(k,l)\cdot(m,n)$ випливає рівність $a-b=k-l+m-n$. Оскільки $(p,q)=(i,j)\cdot(a,b)$, то за твердженням~2.1$(iv)$ з \cite{Fihel-Gutik-2011} маємо, що
\begin{equation*}
p-q=i-j+a-b=i-j+k-l+m-n.
\end{equation*}

$(ii)$ випливає з твердження~\ref{proposition-2.3}.

$(iii)$ $(\Leftarrow)$ Припустимо, що $(x,y)\in{\uparrow_{\preccurlyeq}}(i_1,j_1)$. За твердженням~\ref{proposition-2.3} існує невід'ємне ціле число $a$ таке, що $(x,y)=(i_1-a,j_1-a)$. Тоді
\begin{align*}
  (i_0,i_1)\cdot(i_1-a,j_1-a)\cdot(j_1,j_0)&=(i_0,i_1-(i_1-a)+(j_1-a))\cdot(j_1,j_0)= \\
   &=(i_0,i_1-i_1+a+j_1-a)\cdot(j_1,j_0)= \\
   &=(i_0,j_1)\cdot(j_1,j_0)= \\
   &=(i_0,j_0),
\end{align*}
а отже, $(x,y)$ є розв'язком рівняння \eqref{eq-3.1}.

$(\Rightarrow)$ Припустимо, що $(x,y)$ є розв'язком рівняння \eqref{eq-3.1}. З твердження~$(i)$ випливає, що
\begin{equation*}
  i_0-j_0=i_0-i_1+x-y+j_1-j_0,
\end{equation*}
а отже, $x-y=i_1-j_1$. Тоді за твердженням~$(ii)$ елементи $(x,y)$ і $(i_1,j_1)$ порівняльні стосовно природного часткового порядку $\preccurlyeq$ на напівгрупі $\mathscr{C}_\mathbb{Z}$. Отже, виконується хоча б одна з умов: $(x,y)\preccurlyeq(i_1,j_1)$ або $(i_1,j_1)\preccurlyeq(x,y)$. Якщо $(x,y)\preccurlyeq(i_1,j_1)$, то за твердженням~\ref{proposition-2.3} існує невід'ємне ціле число $a$ таке, що $(x,y)=(i_1+a,j_1+a)$. Тоді
\begin{align*}
  (i_0,i_1)\cdot(i_1+a,j_1+a)\cdot(j_1,j_0)&=(i_0-i_1+(i_1+a),j_1+a)\cdot(j_1,j_0)= \\
   &=(i_0-i_1+i_1+a,j_1+a)\cdot(j_1,j_0)= \\
   &=(i_0+a,j_1+a)\cdot(j_1,j_0)= \\
   &=(i_0+a,j_0+a),
\end{align*}
а отже, $a=0$ i  $(x,y)=(i_1,j_1)\in{\uparrow_{\preccurlyeq}}(i_1,j_1)$. Якщо $(i_1,j_1)\preccurlyeq(x,y)$, то $(x,y)\in{\uparrow_{\preccurlyeq}}(i_1,j_1)$.

Доведення тверджень $(iv)$ і $(v)$ аналогічні доведенню твердження $(iii)$.
\end{proof}

Множину ізольованих точок лівої (правої) топологічної напівгрупи $\mathscr{C}_\mathbb{Z}$ з ізольованою точкою описує твердження \ref{proposition-3.1}.

\begin{proposition}\label{proposition-3.1}
Нехай $\tau$~-- неперервна зліва (справа) $T_1$-топологія на розширеній біциклічній напівгрупі $\mathscr{C}_\mathbb{Z}$. Якщо $(i_0,j_0)$~-- ізольована точка в просторі $(\mathscr{C}_\mathbb{Z},\tau)$, то кожна точка множини $\mathrm{IZ}_{i_0}=\{(i,j)\in\mathscr{C}_\mathbb{Z}\colon i\leqslant i_0 \}$ $(\mathrm{IZ}^{j_0}=\{(i,j)\in\mathscr{C}_\mathbb{Z}\colon j\leqslant j_0 \})$ є ізольованою в просторі $(\mathscr{C}_\mathbb{Z},\tau)$.
\end{proposition}

\begin{proof}
Оскільки праві зсуви в $(\mathscr{C}_\mathbb{Z},\tau)$ неперервні та $(i_0,j_0)$~-- ізольована точка в $(\mathscr{C}_\mathbb{Z},\tau)$, то повний прообраз $\rho_{(j_0,j_0)}^{-1}(i_0,j_0)$ правого зсуву $\rho_{(j_0,j_0)}\colon \mathscr{C}_\mathbb{Z}\to \mathscr{C}_\mathbb{Z}$, $(k,l)\mapsto (k,l)\cdot (j_0,j_0)$ одноточкової множини $\{(i_0,j_0)\}$ є відкрито-замкненою підмножиною в просторі $(\mathscr{C}_\mathbb{Z},\tau)$, оскільки $(\mathscr{C}_\mathbb{Z},\tau)$~-- $T_1$-простір.

З твердження~\ref{proposition-2.3} випливає, що
\begin{equation*}
\rho_{(j_0,j_0)}^{-1}(i_0,j_0)={\uparrow_{\preccurlyeq}}(i_0,j_0)=\{(i_0-k,j_0-k)\colon k\in\omega\}.
\end{equation*}
З того, що $(\mathscr{C}_\mathbb{Z},\tau)$~-- $T_1$-простір, то з вищесказаного випливає, що кожна точка вигляду $(i_0-k,j_0-k)$ є ізольованою в $(\mathscr{C}_\mathbb{Z},\tau)$, оскільки
\begin{equation*}
  \{(i_0-k,j_0-k)\}={\uparrow_{\preccurlyeq}}(i_0,j_0)\setminus\left({\uparrow_{\preccurlyeq}}(i_0{-}k{-}1,j_0{-}k{-}1)\cup \{(i_0,j_0),\ldots,(i_0{-}k{+}1,j_0{-}k{+}1)\}\right).
\end{equation*}

Зафіксуємо довільний елемент $(i,j)$ напівгрупи $\mathscr{C}_\mathbb{Z}$ такий, що $i\leqslant i_0$, тобто $i=i_0-k$ для деякого числа $k\in\omega$. Тоді
\begin{equation*}
  (i_0-k,j_0-k)=(i_0-k,j_0)\cdot (j_0,j_0-k)=\rho_{(j_0,j_0-k)}(i_0-k,j_0).
\end{equation*}
З припущення твердження та вище наведених міркувань випливає, що повний прообраз
\begin{equation*}
  \rho_{(j_0,j_0-k)}^{-1}(i_0-k,j_0-k)={\uparrow_{\preccurlyeq}}(i_0-k,j_0)=\left\{(i_0-k-p,j_0-p)\colon p\in\omega\right\}
\end{equation*}
є відкрито-замненою підмножиною в топологічному просторі $(\mathscr{C}_\mathbb{Z},\tau)$, яка складається з ізольованих у $(\mathscr{C}_\mathbb{Z},\tau)$ точок. Звідси випливає твердження.

Доведення дуального твердження аналогічне.
\end{proof}

Теорема~\ref{theorem-3.2} дає необхідні та достані умови, за виконання яких неперервна зліва (справа) $T_1$-топологія на розширеній біциклічній напівгрупі $\mathscr{C}_\mathbb{Z}$ дискретна.

\begin{theorem}\label{theorem-3.2}
Неперервна зліва (справа) $T_1$-топологія $\tau$ на розширеній біциклічній напівгрупі $\mathscr{C}_\mathbb{Z}$ дискретна тоді і лише тоді, коли існує послідовність $\{(i_n,j_n)\}_{n\in\mathbb{N}}$ ізольованих точок простору $(\mathscr{C}_\mathbb{Z},\tau)$ така, що послідовність $\{i_n\}_{n\in\mathbb{N}}$ $(\{j_n\}_{n\in\mathbb{N}})$ строго монотонно зростаюча.
\end{theorem}

\begin{proof}
Імплікація $(\Rightarrow)$ очевидна.

$(\Leftarrow)$ Припустимо, що існує послідовність $\{(i_n,j_n)\}_{n\in\mathbb{N}}$ ізольованих точок простору $(\mathscr{C}_\mathbb{Z},\tau)$ така, що послідовність $\{i_n\}_{n\in\mathbb{N}}$ $(\{j_n\}_{n\in\mathbb{N}})$ строго монотонно зростаюча. Зафіксуємо довільний елемент $(i,j)$ напівгрупи $\mathscr{C}_\mathbb{Z}$. Тоді існує елмент $(i_k,j_k)$, $k\in\mathbb{N}$, послідовності $\{(i_n,j_n)\}_{n\in\mathbb{N}}$ такий, що $i<i_k$ ($j<j_k$). За твердженням~\ref{proposition-3.1}, $(i,j)$~-- ізольована точка в топологічному просторі $(\mathscr{C}_\mathbb{Z},\tau)$. З вибору точки $(i,j)$ випливає, що топологія $\tau$ дискретна.
\end{proof}

\begin{remark}
З прикладу~\ref{example-2.8} випливає, що умова теореми~\ref{theorem-3.2} про існування послідовності $\{(i_n,j_n)\}_{n\in\mathbb{N}}$ ізольованих точок простору $(\mathscr{C}_\mathbb{Z},\tau)$ суттєва навіть у випадку, коли простір напівгрупи $\mathscr{C}_\mathbb{Z}$ берівський.
\end{remark}

У випадку $T_1$-напівтопологічних напівгруп твердження теореми~\ref{theorem-3.2} послаблюється до такої теореми.

\begin{theorem}\label{theorem-3.3}
Трансляційно-неперервна $T_1$-топологія $\tau$ на розширеній біциклічній напівгрупі $\mathscr{C}_\mathbb{Z}$ дискретна тоді і лише тоді, коли простір  $(\mathscr{C}_\mathbb{Z},\tau)$ містить ізольвану точку.
\end{theorem}

\begin{proof}
Імплікація $(\Rightarrow)$ очевидна.

$(\Leftarrow)$ Припустимо, що простір $(\mathscr{C}_\mathbb{Z},\tau)$ містить ізольвану точку $(i_0,j_0)$. Зафіксуємо довільну точ\-ку $(i,j)\in \mathscr{C}_\mathbb{Z}$. Тоді $(i_0,i)\cdot(i,j)\cdot(j,j_0)=(i_0,j_0)$. За лемою~\ref{lemma-3.1}$(iii)$ множина розв'язків рівняння $(i_0,i)\cdot(x,y)\cdot(j,j_0)=(i_0,j_0)$ збігається з множиною ${\uparrow_{\preccurlyeq}}(i,j)$. Оскільки зсуви в $(\mathscr{C}_\mathbb{Z},\tau)$ непервні, то множина ${\uparrow_{\preccurlyeq}}(i,j)$ є відкритою в просторі $(\mathscr{C}_\mathbb{Z},\tau)$, як повний прообраз відкритої множини $\{(i_0,j_0)\}$ стосовно неперервного відображення, яке є композицією правого $\rho_{(j,j_0)}$ та лівого $\lambda_{(i_0,i)}$ зсувів. За твердженням~\ref{proposition-2.3}, $(i,j)\preccurlyeq(i-1,j-1)$ i ${\uparrow_{\preccurlyeq}}(i,j)\setminus{\uparrow_{\preccurlyeq}}(i-1,j-1)=\{(i,j)\}$. Оскільки $\tau$~--- $T_1$-топологія на $\mathscr{C}_\mathbb{Z}$, то за лемою~\ref{lemma-2.4} множина ${\uparrow_{\preccurlyeq}}(i-1,j-1)$ замкнена в просторі $(\mathscr{C}_\mathbb{Z},\tau)$. Отже, $(i,j)$~-- ізольована точка в просторі $(\mathscr{C}_\mathbb{Z},\tau)$. Із вищевикладеного випливає, що $(\mathscr{C}_\mathbb{Z},\tau)$~-- дискретний простір.
\end{proof}

За твердженням~1.30 з \cite{Haworth-McCoy=1977} кожний зліченний берівський $T_1$-простір міс\-тить нескіченну підмножину ізольових у ньому точок, а отже, з теореми~\ref{theorem-3.3} випливає такий наслідок.

\begin{corollary}\label{corollary-3.4}
Трансляційно-неперервна берівська $T_1$-топологія $\tau$ на розширеній біциклічній напівгрупі $\mathscr{C}_\mathbb{Z}$ дискретна.
\end{corollary}

\begin{lemma}\label{lemma-3.4}
Нехай $\tau$~-- трансляційно-неперервна $T_1$-топологія на розширеній біциклічній напівгрупі $\mathscr{C}_\mathbb{Z}$ така, що відображення $\varphi\colon \mathscr{C}_\mathbb{Z}\to E(\mathscr{C}_\mathbb{Z})$, $x\mapsto xx^{-1}$ i $\psi\colon \mathscr{C}_\mathbb{Z}\to E(\mathscr{C}_\mathbb{Z})$, $x\mapsto x^{-1}x$ неперервні. Якщо для деякого ідемпотента $(i,i)\in\mathscr{C}_\mathbb{Z}$ існує відкритий окіл $U(i,i)$ точки $(i,i)$ у просторі $(\mathscr{C}_\mathbb{Z},\tau)$ такий, що множина $U(i,i)\cap E(\mathscr{C}_\mathbb{Z})$ скінченна, то $\tau$~-- дискретна топологія.
\end{lemma}

\begin{proof}
Позаяк  $\tau$~-- $T_1$-топологія на $\mathscr{C}_\mathbb{Z}$, то не зменшуючи загальності можемо вважати, що існує відкритий окіл точки $(i,i)$ у топологічному просторі $(\mathscr{C}_\mathbb{Z},\tau)$ такий, що $U(i,i)\cap E(\mathscr{C}_\mathbb{Z})=\{(i,i)\}$. З того, що $\tau$~-- трансляційно-неперервна топологія випливає, що для довільного ідемпотента $(j,j)$ напівгрупи $\mathscr{C}_\mathbb{Z}$ існує його відкритий окіл $V(j,j)$ у просторі $(\mathscr{C}_\mathbb{Z},\tau)$ такий, що
\begin{equation*}
  (i,j)\cdot V(j,j)\cdot (j,i)\subseteq U(i,i).
\end{equation*}
З леми~\ref{lemma-3.1}$(iii)$ випливає, що $V(j,j)\cap E(\mathscr{C}_\mathbb{Z})\subseteq {\uparrow_{\preccurlyeq}}(j,j)$, оскільки окіл $U(i,i)$ містить єдиний ідемпотент $(i,i)$. Позаяк $\tau$~-- трансляційно-неперервна $T_1$-топологія, то за лемою~\ref{lemma-2.4} множина ${\uparrow_{\preccurlyeq}}(j-1,j-1)$ замкнена в топологічному просторі $(\mathscr{C}_\mathbb{Z},\tau)$. Отже, відкритий окіл $W(j,j)=V(j,j)\setminus{\uparrow_{\preccurlyeq}}(j-1,j-1)$ містить єдиний ідемпотент напівгрупи $\mathscr{C}_\mathbb{Z}$. Із вищедоведеного випливає, що кожний ідемпотент $(k,k)$ напівгрупи $\mathscr{C}_\mathbb{Z}$ має окіл $W(k,k)$, який містить єдиний ідемпотент .

Нехай $(m,n)$~-- довільний елемент напівгрупи $\mathscr{C}_\mathbb{Z}$. Позаяк відображення $\varphi\colon \mathscr{C}_\mathbb{Z}\to E(\mathscr{C}_\mathbb{Z})$, $x\mapsto xx^{-1}$ i $\psi\colon \mathscr{C}_\mathbb{Z}\to E(\mathscr{C}_\mathbb{Z})$, $x\mapsto x^{-1}x$ неперервні, то з вищевикладеного випливає, що існує відкритий окіл $O(m,n)$ точки $(m,n)$ у топологічному просторі $(\mathscr{C}_\mathbb{Z},\tau)$ такий, що
\begin{equation*}
  (m_1,n_1)\cdot(m_1,n_1)^{-1}=(m_1,m_1)\in \{(m,m)\}
\end{equation*}
і
\begin{equation*}
  (m_1,n_1)^{-1}\cdot(m_1,n_1)=(n_1,n_1)\in \{(n,n)\}
\end{equation*}
для всіх $(m_1,n_1)\in O(m,n)$. Звідси випливає, що множина $O(m,n)$ містить лише елемент $(m,n)$, а отже, виконується твердження леми.
\end{proof}

\begin{lemma}\label{lemma-3.5}
Нехай $\tau$~-- трансляційно-неперервна $T_1$-топологія на розширеній біциклічній напівгрупі $\mathscr{C}_\mathbb{Z}$. Якщо існує елемент $(i,j)\in \mathscr{C}_\mathbb{Z}$ такий, що простір ${\updownarrow_{\preccurlyeq}}(i,j)$~-- квазі-ре\-гу\-ляр\-ний у точці $(i,j)$, то для кожного елемента $(m,n)$ напівгрупи $\mathscr{C}_\mathbb{Z}$ простір ${\updownarrow_{\preccurlyeq}}(m,n)$~-- квазі-регулярний у точці $(m,n)$.
\end{lemma}

\begin{proof}
Позаяк $(\mathscr{C}_\mathbb{Z},\tau)$~-- $T_1$-простір, то  для довільного елемента $(k,l)\in\mathscr{C}_\mathbb{Z}$ множина ${\downarrow_{\preccurlyeq}}(k,l)$~-- відкрита в підпросторі ${\updownarrow_{\preccurlyeq}}(k,l)$, бо за лемою~\ref{lemma-2.4} множина ${\uparrow_{\preccurlyeq}}(k-1,l-1)$ замкнена в $(\mathscr{C}_\mathbb{Z},\tau)$. Відобра\-жен\-ня $\mathfrak{f}_{(i,j)}^{(m,n)}\colon \mathscr{C}_\mathbb{Z}\to\mathscr{C}_\mathbb{Z}$, $(a,b)\mapsto(i,m)\cdot(a,b)\cdot(n,j)$ неперервне як композиція лівого та правого зсувів. За твердженням~\ref{proposition-2.3}, $(m+k,n+k)\in{\uparrow_{\preccurlyeq}}(m,n)$ для довільного $k\in\omega$, і з визначення напівгрупової операції на $\mathscr{C}_\mathbb{Z}$ випливає, що
\begin{align*}
  \mathfrak{f}_{(i,j)}^{(m,n)}(m+k,n+k)&=(i,m)\cdot(m+k,n+k)\cdot(n,j)= \\
   &=(i-m+m+k,n+k)\cdot(n,j)= \\
   &=(i+k,n+k-n+j)= \\
   &=(i+k,j+k)
\end{align*}
для довільного $k\in\omega$. Отже, звуження $\mathfrak{f}_{(i,j)}^{(m,n)}|_{{\downarrow_{\preccurlyeq}}(m,n)}\colon {\downarrow_{\preccurlyeq}}(m,n)\to{\downarrow_{\preccurlyeq}}(i,j)$ i $\mathfrak{f}^{(i,j)}_{(m,n)}|_{{\downarrow_{\preccurlyeq}}(i,j)}\colon $ ${\downarrow_{\preccurlyeq}}(i,j)\to{\downarrow_{\preccurlyeq}}(m,n)$ відображень $\mathfrak{f}_{(i,j)}^{(m,n)}$ i $\mathfrak{f}^{(i,j)}_{(m,n)}$, відповідно, є взаємно оберненими від\-об\-раженнями, а з неперервності зсувів у $(\mathscr{C}_\mathbb{Z},\tau)$ випливає, що $\mathfrak{f}_{(i,j)}^{(m,n)}|_{{\downarrow_{\preccurlyeq}}(m,n)}$ i $\mathfrak{f}^{(i,j)}_{(m,n)}|_{{\downarrow_{\preccurlyeq}}(i,j)}$ є гомеоморфізмами просторів ${\downarrow_{\preccurlyeq}}(i,j)$ i ${\downarrow_{\preccurlyeq}}(m,n)$. Позаяк для довільного елемента $(p,q)$ напівгрупи $\mathscr{C}_\mathbb{Z}$ множина ${\downarrow_{\preccurlyeq}}(p,q)$ відкрита в просторі ${\updownarrow_{\preccurlyeq}}(p,q)$, то отримуємо твердження леми.
\end{proof}

\begin{lemma}\label{lemma-3.6}
Нехай $\tau$~-- інверсна напівгрупова $T_1$-топологія на розширеній біциклічній напівгрупі $\mathscr{C}_\mathbb{Z}$. Якщо існує ідемпотент $(i,i)\in\mathscr{C}_\mathbb{Z}$ такий, що топологічний простір $E(\mathscr{C}_\mathbb{Z})$ квазі-регулярний у точці $(i,i)$, то $\tau$~-- дискретна топологія.
\end{lemma}

\begin{proof}
Нехай $U(i,i)$~-- відкритий окіл точки $(i,i)$ в просторі $E(\mathscr{C}_\mathbb{Z})$. Якщо $\tau$~-- недискретна топологія на $\mathscr{C}_\mathbb{Z}$, то з  леми~\ref{lemma-3.4} випливає, що $U(i,i)$~-- нескінченна множина. Позаяк $\tau$~-- $T_1$-топологія, то $U_{(i,i)}=U(i,i)\setminus\{(i,i)\}$~-- відкрита непорожня підмножина в просторі $E(\mathscr{C}_\mathbb{Z})$. Тоді існує непорожня відкрита підмножина $W_{(i,i)}\subseteq U_{(i,i)}$ така, що $\operatorname{cl}_{E(\mathscr{C}_\mathbb{Z})}(W_{(i,i)})\subseteq U_{(i,i)}$. Очевидно, що $O(i,j)=U(i,j)\setminus \operatorname{cl}_{E(\mathscr{C}_\mathbb{Z})}(W_{(i,i)})$~-- відкритий окіл точ\-ки $(i,i)$ в просторі $E(\mathscr{C}_\mathbb{Z})$. Не зменшуючи загальності можемо вважати, що множина $W_{(i,i)}$ нескіченна. Справді, у протилежному випадку існує ідемпотент $(j,j)\in\mathscr{C}_\mathbb{Z}$, який має відкритий окіл $U(j,j)$, що задовольняє умову $U(j,j)\cap E(\mathscr{C}_\mathbb{Z})=\{(j,j)\}$, а тоді, за лемою~\ref{lemma-3.4}, $\tau$~-- дискретна топологія.

За лемою~\ref{lemma-2.4} множина ${\uparrow_{\preccurlyeq}}(i,i)$ замкнена в топологічному просторі $(\mathscr{C}_\mathbb{Z},\tau)$, а отже, і в просторі $E(\mathscr{C}_\mathbb{Z})$. Звідси випливає, що існує ідемпотент $(j,j)\in W_{(i,i)}$ такий, що $(j,j)\in {\downarrow_{\preccurlyeq}^\circ}(i,i)$. Тоді $(i,i)\cdot(j,j)=(j,j)$ і з неперервності напівгрупової операції в $(\mathscr{C}_\mathbb{Z},\tau)$ випливає, що існують відкриті околи $W_1(i,i)$ й $W_1(j,j)$ точок $(i,i)$ та $(j,j)$, відповідно, в просторі $(\mathscr{C}_\mathbb{Z},\tau)$ такі, що
\begin{equation}\label{eq-3.2}
  (W_1(i,i)\cap E(\mathscr{C}_\mathbb{Z}))\cdot (W_1(j,j)\cap E(\mathscr{C}_\mathbb{Z}))\subseteq W_{(i,i)},
\end{equation}
\begin{align*}
  W_1(i,i)\cap E(\mathscr{C}_\mathbb{Z})&\subseteq O(i,i), \\
  W_1(j,j)\cap E(\mathscr{C}_\mathbb{Z})&\subseteq W_{)i,i)},
\end{align*}
і множини $W_1(i,i)\cap E(\mathscr{C}_\mathbb{Z})$ та $W_1(j,j)\cap E(\mathscr{C}_\mathbb{Z})$ нескінченні. Отже, для довільного ідемпотента $(k,k)\in W_1(j,j)$ існує ідемпотент $(l,l)\in W_1(i,i)$ такий, що $(k,k)\cdot(l,l)=(l,l)\cdot(k,k)=(l,l)$, а це суперечить умові \eqref{eq-3.2}. Звідси випливає, що хоча б одна з множин $W_1(i,i)\cap E(\mathscr{C}_\mathbb{Z})$ або $W_1(j,j)\cap E(\mathscr{C}_\mathbb{Z})$ скінченна. Далі скористаємося лемою~\ref{lemma-3.4}.
\end{proof}

З лем~\ref{lemma-3.5} i~\ref{lemma-3.6} випливає теорема~\ref{theorem-3.7}.

\begin{theorem}\label{theorem-3.7}
Нехай $\tau$~-- інверсна напівгрупова $T_1$-топологія на розширеній біциклічній напівгрупі $\mathscr{C}_\mathbb{Z}$. Якщо існує елемент $(i,j)\in\mathscr{C}_\mathbb{Z}$ такий, що топологічний простір ${\updownarrow_{\preccurlyeq}}(i,j)$ квазі-регулярний у точці $(i,j)$, то $\tau$~-- дискретна топологія.
\end{theorem}


\end{document}